\documentclass[12pt]{article}
\usepackage{textcomp}
\usepackage{amssymb}
\usepackage{latexsym}
\usepackage{amsthm}
\usepackage{amscd}
\usepackage{amsmath}
\usepackage{mathrsfs}
\usepackage{amsfonts}
\usepackage{enumitem}

\usepackage{lineno}
\usepackage{a4wide}
\usepackage{amsmath}
\usepackage{amssymb}
\usepackage{amsthm}
\usepackage{latexsym}
\usepackage{graphicx}
\usepackage[english]{babel}
\usepackage{makeidx}

\newtheorem{thm}{Theorem}[section]
\newtheorem{lem}[thm]{Lemma}
\newtheorem{cor}[thm]{Corollary}
\newtheorem{defn-lem}[thm]{Definition-Lemma}
\newtheorem{conj}[thm]{Conjecture}

\newtheorem{prop}[thm]{Proposition}

\theoremstyle{remark}

\theoremstyle{definition}

\numberwithin{equation}{section}

\allowdisplaybreaks[4]

\def \N{{\mathbb N}}

\def \Z{{\mathbb Z}}
\def \R{{\mathbb R}}

\def\map#1.#2.{#1 \longrightarrow #2}
\def\rmap#1.#2.{#1 \dasharrow #2}

\DeclareMathOperator{\depth}{depth}
\DeclareMathOperator{\sdepth}{sdepth}
\def\fb#1.{\underset #1 \to \times}
\def\pr#1.{\Bbb P^{#1}}
\def\ring#1.{\mathcal O_{#1}}
\def\mlist#1.#2.{{#1}_1,{#1}_2,\dots,{#1}_{#2}}

\def\uloopr#1{\ar@'{@+{[0,0]+(-4,5)} @+{[0,0]+(0,10)}
@+{[0,0]+(4,5)}}^{#1}}

\def\dloopr#1{\ar@'{@+{[0,0]+(-4,-5)} @+{[0,0]+(0,-10)}
@+{[0,0]+(4,-5)}}_{#1}}

\def\rloopd#1{\ar@'{@+{[0,0]+(5,4)} @+{[0,0]+(10,0)}
@+{[0,0]+(5,-4)}}^{#1}}

\newcommand{\rdots}{\mathinner{\mkern1mu\raise1pt\hbox{.}\mkern2mu\raise4pt\hbox{.} \mkern2mu\raise7pt\vbox{\kern7pt\hbox{.}}\mkern1mu}} 

\def\lloopd#1{\ar@'{@+{[0,0]+(-5,4)} @+{[0,0]+(-10,0)}
@+{[0,0]+(-5,-4)}}_{#1}}

\long\def\ignore#1{}
\long\def\ignore#1{#1}

\begin{document}

\begin{center}
{\bf On a conjecture of Stanley depth of squarefree Veronese ideals}
\end{center}

\medskip

\begin{center}  
{{\small Maorong Ge, Jiayuan Lin and Yi-Huang Shen}}
\end{center}

{\small {\bf Abstract} In this paper, we partially confirm a conjecture, proposed by Cimpoea\c{s}, Keller, Shen, Streib and Young, on the Stanley depth of squarefree Veronese ideals $I_{n,d}$. This conjecture suggests that, for positive integers $1 \le d \le n$, $\sdepth (I_{n,d})=\left \lfloor \binom{n}{d+1}/\binom{n}{d} \right\rfloor+d$. Herzog, Vladoiu and Zheng established a connection between the Stanley depths of quotients of monomial ideals and interval partitions of certain associated posets. Based on this connection, Keller, Shen, Streib and Young recently developed a useful combinatorial tool to analyze the interval partitions of the posets associated with the squarefree Veronese ideals. We modify their ideas and prove that if $1 \le d \le n \le (d+1) \left \lfloor \frac{1+\sqrt{5+4d}}{2}\right\rfloor+2d$, then $\sdepth (I_{n,d})=\left \lfloor \binom{n}{d+1}/\binom{n}{d} \right\rfloor+d$. We also obtain $\left \lfloor \frac{d+\sqrt{d^2+4(n+1)}}{2} \right\rfloor \le \sdepth(I_{n,d})  \le \left \lfloor \binom{n}{d+1}/\binom{n}{d} \right\rfloor+d$ for $n > (d+1) \left \lfloor \frac{1+\sqrt{5+4d}}{2}\right\rfloor+2d$. As a byproduct of our construction, We give an alternative proof of Theorem $1.1 $ in $[13]$ without graph theory.}

\section{Introduction}

The concept of Stanley depth was first introduced by Stanley in $[20]$. Let us briefly recall its definition here.

Let $S=K[x_1, \cdots, x_n]$ be the naturally $\Z^n$-graded polynomial ring in $n$ variables over a field $K$. A $\textit{Staney decomposition}$ of a finitely generated $\Z^n$-graded $S$-module $M$ is a finite direct sum of $K$-vector spaces

$$\mathscr{D}: M=\overset{m}{\underset{i=1}{\bigoplus}} u_i K[Z_i]$$

\noindent where each $u_i \in M$ is homogeneous, and $Z_i$ is a subset of $\{x_1, \cdots, x_n\}$. The number $\sdepth \mathscr{D}=\min\{|Z_i|: i=1, \cdots, m\}$ is called the $\textit{Staney depth of}$ $\mathscr{D}$. The $\textit{Staney depth of}$ $M$ is defined to be

$$\sdepth M=\max\{\sdepth \mathscr{D}: \mathscr{D} \hskip .1 cm \text{is a Stanley decomposition of M}\}.$$

In $[20]$, Stanley conjectures that $\depth M \le \sdepth M$ for all finitely generated $\Z^n$-graded $S$-module $M$. Although this conjecture remains open in general, it has been confirmed in several special cases, see, for example, $[1$-$3]$, $[5$-$12]$, $[17]$ and $[18]$. In $[11]$, Herzog, Vladoiu and Zheng proved that the Stanley depth of $M=I/J$ can be computed in finite number of steps, where $J \subset I$ are monomial ideals of $S$. They associated $I/J$ with a poset $P_{I/J}$ and showed that the Stanley depth of $I/J$ is determined by partitioning $P_{I/J}$ into suitable intervals. Since their pioneer work, some progress has been made in calculating the Stanley depths of $I$ and $S/I$. See, for instance, $[4]$, $[13$-$16]$ and $[19]$.

In this paper, we investigate $\textit{squarefree Veronese ideal}$ $I_{n,d}$ generated by all squarefree monomials of degree $d$ in $S$. Cimpoea\c{s} $[7, {\text Conjecture \hskip 0.05 cm 1.6}]$ and Keller, Shen, Streib and Young $[13, {\text Conjecture \hskip .05 cm 2.4}]$ conjecture that

\begin{conj} For positive integers $1 \le d \le n$, $\sdepth(I_{n,d})=\left \lfloor \binom{n}{d+1}/\binom{n}{d} \right\rfloor+d$.
\end{conj}

A special case of Conjecture $1.1$, when $d=1$, was first proposed by Herzog et al. in $[11]$. It was settled by Bir\'o, Howard, Keller, Trotter and Young in $[4]$. Cimpoea\c{s} confirmed the above conjecture for $2d+1 \le n \le 3d$. Both Cimpoea\c{s} $[7]$ and Keller, Shen, Streib and Young $[13]$ obtained an upper bound $\sdepth(I_{n,d}) \le \left \lfloor \binom{n}{d+1}/\binom{n}{d} \right\rfloor+d$.  Based on the result in $[4]$, Keller, Shen, Streib and Young $[13]$ proved the above conjecture for $1 \le d \le n <5d+4$ and obtained a lower bound $\sdepth(I_{n,d}) \ge d+3$ if $n \ge 5d+4$. In this paper, we modify their ideas and prove the following theorem.

\begin{thm}
Let $S=K[x_1, \cdots, x_n]$ be the polynomial ring in $n$ variables over a field $K$. Let $I_{n,d}$ be the squarefree Veronese ideal generated by all squarefree monomials of degree $d$ in $S$. 

\begin{description}
\item(1) If $1 \le d \le n \le (d+1) \left \lfloor \frac{1+\sqrt{5+4d}}{2}\right\rfloor+2d$, then the Stanley depth $\sdepth(I_{n,d})=\left \lfloor \binom{n}{d+1}/\binom{n}{d} \right\rfloor+d$.

\item(2) If  $n > (d+1) \left \lfloor \frac{1+\sqrt{5+4d}}{2}\right\rfloor+2d$, then 
$\left \lfloor \frac{d+\sqrt{d^2+4(n+1)}}{2} \right\rfloor \le \sdepth(I_{n,d})  \le \left \lfloor \binom{n}{d+1}/\binom{n}{d} \right\rfloor+d$. 
\end{description}

\end{thm}  

This paper is organized as follows: in section $2$, we review the method of Herzog et al. for associating a poset with a squarefree monomial ideal $I$ and recall some combinatorial tools used in $[13]$. In section 3, we modify their ideas and construct intervals through higher circular representation. In section $4$, we prove the Theorem $1.2$. In section $5$ we give an alternative proof of Theorem $1.1$ in $[13]$ without using graph theory.  

{\bf Acknowledgments} We would like to express our gratitude to Dr. Stephen J. Young for pointing out a gap in the first version of this paper.

\section{Interval partitions and preliminary results}

\subsection{Poset associated with a squarefree monomial ideal $I$}

For a positive integer $n$, let $[n]=\{1, 2, \cdots, n\}$. Let $\N$ be the set of non-negative integers. For each ${\bf c}=(c(1), \cdots, c(n)) \in \N^n$, denote ${\bf x^c}=\prod_i x_i^{c(i)}$. The monomial ${\bf x^c}=\prod_i x_i^{c(i)}$ is squarefree when $c(i)=0$ or $1$ for $1 \le i \le n$. When ${\bf x^c}$ is squarefree, denote its support by $supp ({\bf x^c})=\{i \big\vert c(i)=1\} \subset [n]$. Let $P_{I}$ be the poset consisting of all the supports of squarefree monomials in $I$ and their supersets in $[n]$. It is a Boolean subalgebra of subsets of $[n]$ partially ordered by inclusion. For every $A, B \in P_{I}$ with $A \subseteq B$, define the interval $[A, B]$ to be $\{C \big\vert A \subseteq C \subseteq B\}$.

Let $\mathscr{P}: P_{I}= \cup_{i=1}^r [A_i, B_i]$ be a partition of $P_{I}$ and ${\bf a_i} \in \N^n$ be the tuples such that $Supp ({\bf x^{a_i}})=A_i$. Then there is a Stanley decomposition $\mathscr{D(P)}$ of $I$:

$$\mathscr{D(P)}: I=\overset{r}{\underset{i=1}{\bigoplus}} {\bf x^{a_i}} K[\{x_j \big\vert j \in B_i\}].$$

The $\sdepth ({\mathscr{D(P)}})$ is $\min \{|B_1|, \cdots, |B_r|\}$. Moreover, Herzog et al. showed in $[11]$ that if $I$ is a squarefree monomial ideal, then

$$\sdepth (I)=\max \hskip .1 cm \{\sdepth (\mathscr{D(P)}) \hskip .1 cm \big\vert \hskip .1 cm \mathscr{P} \hskip .1 cm \text{is a partition of} \hskip .1 cm P_I\}.$$

By applying this connection, in $[7, {\text Theorem \hskip .05 cm 1.1}]$ and $[13, {\text Lemma \hskip .05 cm 2.2}]$, the authors proved $\sdepth(I_{n,d}) \le \left \lfloor \binom{n}{d+1}/\binom{n}{d} \right\rfloor+d$. If we can prove that there is a partition of $\mathscr{P}: P_I= \cup_{i=1}^r [A_i, B_i]$ such that $|B_i| \ge \left \lfloor \binom{n}{d+1}/\binom{n}{d} \right\rfloor+d$, Conjecture $1.1$ follows immediately.

Before we proceed to the proof of Theorem $1.2$, let us recall some combinatorial tools developed in $[13]$.

\subsection{Block Structures on $[n]$}

Given a positive integer $n$, we can evenly distribute the points $1, 2, \dots, n$ in clockwise direction around a circle in the plane. In $[13]$, this arrangement is called the circular representation of $[n]$. Given the circular representation of $[n]$, a block is a subset of consecutive points on the circle. For $i, j \in [n]$ we denote by $[i,j]$ the block starting at $i$ and ending at $j$ when traversing the circular representation of $[n]$ clockwise. Given a subset $A \subseteq [n]$ and a density $\delta \ge 1$, the $\textit{block structure of }$ $A$ $\textit{with respect to }$ $\delta$ is a partition of the elements of the circular representation of $[n]$ into clockwise-consecutive blocks $B_1, G_1, B_2, G_2, \cdots, B_p, G_p$ such that 

\begin{description} [itemsep=0pt, parsep=0pt]
\item(i) the first (going clockwise) element $b_i$ of $B_i$ is in $A$;
\item(ii) for all $i \in [p]$, $G_i \cap A=\emptyset$;
\item(iii) for all $i \in [p]$, $\delta \cdot |A \cap B_i|-1<|B_i| \le \delta \cdot |A \cap B_i|$;
\item(iv) for all $y \in B_i$ such that $[b_i,y] \subsetneq B_i$, $|[b_i,y]|+1 \le \delta \cdot |[b_i,y] \cap A|$.
\end{description}
 
The following lemma was proved in $[13]$.

\begin{lem} (Lemma $2.7$ in $[13]$) For $1 \le \delta \le (n-1)/|A|$, the block structure for a set $A$ with respect to $\delta$ on $[n]$ exists and is unique.
\end{lem}

We denote the set $\{B_1, B_2, \cdots, B_p\}$ by blocks$_{\delta}(A)$ and the union $B_1 \cup B_2 \cup \cdots \cup B_p$ by $\mathscr{B}_\delta (A)$. Each $G_i$ is called a gap. We denote the set $\{G_1, G_2, \cdots, G_p\}$ by gaps$_{\delta}(A)$ and the union $G_1 \cup G_2 \cup \cdots \cup G_p$ by $\mathscr{G}_\delta (A)$. Given a density $\delta$, let $f_\delta$ be the function that maps each $A \subseteq [n]$ with $|A| \le (n-1)/\delta$ to $A \cup \mathscr{G}_\delta (A) \subseteq [n]$. Throughout this paper we will focus on intervals of the form $[A, f_\delta (A)]$, or its extended version that we will explain in Section $3$.

We also need the following lemmas in $[13]$.

\begin{lem} (Lemma $3.1$ in $[13]$) Given a positive integer $n$, let $A, A' \subseteq [n]$ with $A \ne A'$ and $|A|=|A'|$, and let $\delta \ge 1$. If $|f_\delta(A)|-|A| \le \delta-1$, then $[A, f_\delta(A)]$ does not intersect $[A', f_\delta(A')]$.
\end{lem}

\begin{lem} (Lemma $3.2$ in $[13]$) Let $d$ be a positive integer and $k$ a nonnegative integer and let $n=(d+1)k+d$. Let $A \subseteq [n]$ be a $d$-set. Then $|f_{k+1}(A)|=d+k$.
\end{lem}

\begin{lem} (Lemma $3.5$ in $[13]$) Let $d, k$ and $l$ be positive integer, $n=(d+1)k+d$ and $\mathscr{I}_{n,d, k+1}=\{[A, f_{k+1}(A)] \big\vert A \subseteq [n], |A|=d\}$. Suppose $D_l$ is a $(d+l)$-set that is not covered by any element of $\mathscr{I}_{n,d, k+1}$. Then there is no superset of $D_l$ that is covered by an element of $\mathscr{I}_{n,d, k+1}$.
\end{lem}

In the next section, we introduce the higher circular representation that extends these results.

\section{Constructing intervals through higher circular representation}

For any $0 \le l <k$, let $s$ be a fixed positive integer less than or equal to $\left \lfloor \frac{n-d-l}{d+l+1}\right\rfloor$.  Then the integer $m=(n+1)s+n$ satisfies the following properties:

\begin{lem}
(1) $m>n$, (2) $0 \le s+1 \le \frac{m-1}{n}$, and (3) $\frac{m-n}{s+1} \le n-(d+l) < m-n$.
\end{lem}
\begin{proof}
(1) and (2) follow trivially from the definition of $m$.

By the definition of $m$, we have $m-n=(n+1)s >n >n-(d+l)$.

The left inequality in (3) is equivalent to $m-n=(n+1)s \le (s+1)n- (d+l)(s+1)$ or $(1+d+l)s \le n-d-l$. The latter one follows from $s \le \left \lfloor \frac{n-d-l}{d+l+1}\right\rfloor$.
\end{proof}

Considering the circular representation of $[m]$ and applying Lemma $2.1-2.4$ with $\delta=s+1$, we have the following corollary.

\begin{cor} For each $n$-set $A \subseteq [m]$, $|f_{s+1}(A)|=n+s$. All the intervals in the set

$$\mathscr{I}_{m,n, s+1}=\{[A, f_{s+1}(A)] \big\vert A \subseteq [m], |A|=n\}$$

\noindent are disjoint. Moreover, if $D_l$ is a $(n+l)$-set that is not covered by any element of $\mathscr{I}_{m,n, s+1}$, then there is no superset of $D_l$ that is covered by an element of $\mathscr{I}_{m,n, s+1}$.
\end{cor}

Denote $F_{n,d+l}=\{A \subset [n] \big\vert |A|=d+l\}$ for $0 \le l <k$. We have that

\begin{prop} For any $A \in F_{n, d+l}$, let $\widetilde{A}=A \cup \{n+1,\cdots, n+n-(d+l)\} \subseteq [m]$. Then $|f_{s+1}(\widetilde{A}) \cap [n]|=d+l+s$ and the intervals in the set
$$\mathscr{I}_{n,d+l, s+1}=\{[A, f_{s+1}(\widetilde{A}) \cap [n] \hskip .1 cm ] \hskip .1 cm \big\vert \hskip .1 cm A \in F_{n, d+l}\}$$
\noindent are disjoint. Moreover, if $D_{l'} \subset [n]$ is a $(d+l+l')$-set that is not covered by any element of $\mathscr{I}_{n,d+l,s+1}$, then there is no superset of $D_{l'}$ that is covered by an element of $\mathscr{I}_{n,d+l,s+1}$.
\end{prop}
\begin{proof}
For any $A \in F_{n,d+l}$, the set $\widetilde{A}=A \cup \{n+1,\cdots, n+n-(d+l)\}$ is an $n$-set in $[m]$. Because the block structure of $\widetilde{A}$ with respect to density $s+1$ exists and is unique by Lemma $2.1$ and Lemma $3.1(2)$, it is not hard to see that the consecutive integers $n+1, \cdots, m$ must lie in one block by Lemma $3.1 (3)$. So none of those numbers appears in $f_{s+1} (\widetilde{A}) \setminus \widetilde{A}$. This fact, together with $|f_{s+1} (\widetilde{A})|=n+s$, implies that $|f_{s+1} (\widetilde{A}) \cap [n]|=n+s-[n-(d+l)]=d+l+s$.

If the intersection of two intervals $[A, f_{s+1}(\widetilde{A}) \cap [n]]$ and $[A', f_{s+1}(\widetilde{A}') \cap [n]]$ in $\mathscr{I}_{n,d+l,s+1}$ is nonempty, then there exists a subset $B \in [n]$ such that $A \subseteq B \subseteq f_{s+1}(\widetilde{A}) \cap [n]$ and $A' \subseteq B \subseteq f_{s+1}(\widetilde{A}') \cap [n]$. Because $f_{s+1}(\widetilde{A})$ and $f_{s+1}(\widetilde{A}')$ both contains $\{n+1,\cdots, n+n-(d+l)\}$, so $\widetilde{B}=B \cup \{n+1,\cdots, n+n-(d+l)\}$ is a common element of $[\widetilde{A}, f_{s+1}(\widetilde{A})]$ and $[\widetilde{A}', f_{s+1}(\widetilde{A}')]$, a contradiction! Therefore the intervals in the set
$\mathscr{I}_{n,d+l, s+1}=\{[A, f_{s+1}(\widetilde{A}) \cap [n]] \big\vert A \in F_{n, d+l}\}$ are disjoint.

Suppose that $D_{l'} \subset [n]$ is a $(d+l+l')$-set that is not covered by any element of $\mathscr{I}_{n,d+l,s+1}$. We prove that there is no superset of $D_{l'}$ that is covered by an element of $\mathscr{I}_{n,d+l,s+1}$. Otherwise, there is a superset of $D_{l'}$, say $D$, is covered by $[A, f_{s+1}(\widetilde{A}) \cap [n]]$ for some $A \in F_{n, d+l}$. Because $f_{s+1}(\widetilde{A}) \cap [n]$ is also a superset of $D_{l'}$ covered by an interval in $\mathscr{I}_{n,d+l, s+1}$, it is sufficient to assume that $D=f_{s+1}(\widetilde{A}) \cap [n]$. We follow the proof of Lemma $3.5$ in $[13]$ to get a contradiction.

For any $[A, f_{s+1}(\widetilde{A}) \cap [n]] \in \mathscr{I}_{n,d+l, s+1}$ with $D_{l'} \subseteq f_{s+1}(\widetilde{A}) \cap [n]$, let $X=(f_{s+1}(\widetilde{A}) \cap [n]) \setminus D_{l'}=f_{s+1}(\widetilde{A}) \setminus \widetilde {D_{l'}}$. We call such a combination of sets $(X,A)$ a $\textit{pair}$. We call the pair $(X,A)$ $\textit{optimal}$ if, among all pairs, $|X \cap A|$ is minimized.

Let $(X^{(0)},A^{(0)})$ be an optimal pair. Notice that if $X^{(0)} \cap A^{(0)} =\emptyset$, then $A^{(0)} \subseteq f_{s+1}(\widetilde{A^{(0)}}) \cap [n] \subseteq X^{(0)} \cup D_{l'}$ implies $A^{(0)} \subseteq D_{l'}$. So $D_{l'}$ is covered by $[A^{(0)}, f_{s+1}(\widetilde{A^{(0)}}) \cap [n]]$, a contradiction. Thus $|X^{(0)} \cap A^{(0)}| \ge 1$. Consider $x_0 \in X^{(0)} \cap A^{(0)}$. Let $B \in$ blocks$_{(s+1)}(\widetilde{A^{(0)}})$ such that $x_0 \in B$. Let $x_1$ be the first element in $\mathscr{G}_{s+1} (\widetilde{A^{(0)}})$ counterclockwise from $B$ in the circular representation of $[m]$, and let $z_0$ be the last element of $B \setminus \{n+1, \cdots, m\}$ (that is, the most clockwise element of $B$ which is not in the set $\{n+1, \cdots, m\}$). Because the set $\{n+1, \cdots, m\}$ lie in a single block in blocks$_{(s+1)}(\widetilde{A^{(0)}})$, $\{x_1, z_0\} \subseteq [n]$. The point $z_0$ exists and is distinct from $x_0$ since the density $(s+1) \ge 2$ and a block cannot end with an element of $A^{(0)}$. Let $x_2, \cdots, x_p$ be the successive elements of the gaps of $\mathscr{G}_{s+1} (\widetilde{A^{(0)}})$, indexed counterclockwise from $x_1$. Fix $q$ as small as possible so that $x_{q+1} \in D_{l'}$. Such a $q$ must exist, as otherwise all the gap points belong to
$X^{(0)}$, so we would have $|X^{(0)}| \ge s$, but $|X^{(0)}|=d+l+s-(d+l+l')=s-l'<s$.

We now define a sequence of pairs $(X^{(i)},A^{(i)})$. For every $i$, $0 \le i \le q$, let $A^{(i+1)}=\left (A^{(i)} \setminus \{x_i\}\right ) \cup \{x_{i+1}\}$. Notice that $f_{s+1}(\widetilde{A^{(i+1)}}) \cap [n]=\left ((f_{s+1} (\widetilde{A^{(i)}}) \cap [n]) \setminus \{x_i\}\right ) \cup \{z_i\}$, where $z_i$ is the last element in $B_i \setminus \{n+1, \cdots, m\}$ and $B_i$ is the block in blocks$_{(s+1)}(\widetilde{A^{(i)}})$ that contains $x_i$. We define $X^{(i+1)}=\left (f_{s+1}(\widetilde{A^{(i+1)}}) \cap [n] \right ) \setminus D_{l'}=(X^{(i)} \setminus \{x_i\}) \cup \{z_i\}$. Now we have the following property from our definition

$|X^{(i+1)} \cap A^{(i+1)}|=|\left((X^{(i)} \setminus \{x_i\}) \cup \{z_i\} \right ) \cap \left((A^{(i)} \setminus \{x_i\}) \cup \{x_{i+1}\} \right) |$

\hskip 2.98 cm $=|\left((X^{(i)} \cap A^{(i)}) \setminus \{x_i\} \right) \cup \left (\{z_i\} \cap A^{(i)} \right) \cup \left(X^{(i)} \cap \{x_{i+1}\} \right)|$

\hskip 2.98 cm $=|A^{(i)} \cap X^{(i)}|-1+0+|X^{(i)} \cap \{x_{i+1}\}|$

For $0 \le i<q$, we know that $x_{i+1} \notin D_{l'}$, and hence $x_{i+1} \in X^{(i)}$. From the computation above, we see that this implies $|X^{(i+1)} \cap A^{(i+1)}|=|X^{(i)} \cap A^{(i)}|$. Thus, the pair $(X^{(i+1)}, A^{(i+1)})$ is optimal for $i<q$. On the other hand, $x_{q+1} \in D_{l'}$, and then $x_{q+1} \notin X^{(i)}$. So we have $|X^{(q+1)} \cap A^{(q+1)}|<|X^{(0)} \cap A^{(0)}|$, contrary to the optimality of $(X^{(0)}, A^{(0)})$. This completes the proof.
\end{proof}

The following disjointness result is a generalization of Lemma $3.1$ in $[13]$. Although we do not use it directly in this paper, it gives us some ideas on how to compare intervals with different densities.

\begin{prop}
Given a positive integer $n$, let $A, B \subseteq [n]$ with $|A| \le |B|$, and let $\delta, \eta \in \R$ with $\delta \ge \eta \ge 1$, $\delta \cdot |A| \le n-1$ and $\eta \cdot |B|  \le n-1$. If $|f_{\eta} (B)|-|B| \le \eta -1$ and $A \nsubseteq B$, then $[A, f_{\delta}(A)]$ does not intersect $[B, f_{\eta}(B)]$.
\end{prop}

\begin{proof}
Suppose for contradiction that $[A, f_{\delta}(A)]$ and $[B, f_{\eta}(B)]$ have a nontrivial intersection, then $A \subseteq B \cup \mathscr{G}_{\eta}(B)$ and $B \subseteq A \cup \mathscr{G}_{\delta}(A)$. Because $A \nsubseteq B$, there exists an $a \in A \setminus B$ satisfying $a \in \mathscr{G}_{\eta}(B)$. Let $B_1, G_1, B_2, \cdots, B_p, G_p$ be the block structure of $B$ with density $\eta$ and, without loss of generality, suppose $a \in G_1$. Note that

$$|\mathscr{G}_{\eta}(B)|=|f_{\eta} (B)|-|B| \le \eta -1 \le \delta -1.$$

Therefore, the block in blocks$_{\delta}(A)$ that contains $a$, say $A'$, must also contain $b_{2_1}$, the first element of $B_2$. As $b_{2_1} \in B$ and $b_{2_1} \notin \mathscr{G}_{\delta} (A)$, it must be that $b_{2_1} \in A$. Let $b_{2_2}$ be the next element of $B_2 \cap B$ found when proceeding clockwise around the circular representation of $[n]$. Clearly $b_{2_2} \in A'$, so $b_{2_2} \in A$ as otherwise $b_{2_2} \notin f_{\delta}(A)$. Proceeding clockwise in this manner, we find that $B \cap B_2 \subseteq A$ and $B_2 \subseteq A'$.

Again, using the fact that $|\mathscr{G}_{\eta}(B)| \le \eta -1 \le \delta -1$ and the fact that $a \in A'$, we find that $b_{3_1}$, the first element of $B_3$, is also in $A'$. Applying the same argument as was applied to $B_2$ we find that $B \cap B_3 \subseteq A$ and $B_3 \subseteq A'$.

Finally, proceeding clockwise and using analogous arguments, we conclude that $B=B \cap \mathscr{B}_{\eta} (B) \subseteq A$. Since $|A| \le |B|$, this implies that $A=B$, contradicting the assumption that $A \nsubseteq B$. Therefore $[A, f_{\delta}(A)]$ does not intersect $[B, f_{\eta}(B)]$.
\end{proof}

Let $C$ and $D$ be two subsets of $[n]$ with $|C|=d+q$, $|D|=d+l$ and $0 \le q \le l<k$. Let $m'=(n+1) (\eta-1) +n$ and $m''=(n+1) (\delta-1) +n$, where $\eta$ and $\delta$ are two positive integers with $\eta \le \left \lfloor \frac{n+1}{d+l+1}\right\rfloor$ and $\delta \le \left \lfloor \frac{n+1}{d+q+1}\right\rfloor$. Let $\widetilde{C}=C \cup \{n+1, \cdots, n+n-(d+q)\} \subseteq [m'']$ and  $\widetilde{D}=D \cup \{n+1, \cdots, n+n-(d+l)\} \subseteq [m']$. If $\eta \le \delta$, then $m' \le m''$. From Proposition $3.3$ and its proof we know that the consecutive integers $\{n+1, \cdots, m'\}$ must lie in a single block in blocks$_{\eta}(\widetilde{D})$ of $[m']$. We can replace this block with the union of itself and the consecutive numbers $\{m'+1, \cdots, m''\}$ and keep the remaining blocks and gaps in the block structure of $\widetilde{D}$ unchanged. The resulting block structure of $\widetilde{D}$ on $[m'']$ is called the extended block structure. We have the following proposition.

\begin{prop}
Given $C, D \subseteq [n]$ with $|C|=d+q$, $|D|=d+l$ and $0 \le q \le l<k$. Suppose that there exist two positive integers $\eta$ and $\delta$ such that $\eta \le \delta$, $\eta \le \left \lfloor \frac{n+1}{d+l+1}\right\rfloor$, $\delta \le \left \lfloor \frac{n+1}{d+q+1}\right\rfloor$ and $(d+l+1) \eta \ge (d+q+1) \delta$. Then $[C, f_{\delta}(\widetilde{C}) \cap [n]]$ does not intersect $[D, f_{\eta}(\widetilde{D}) \cap [n]]$ if $D$ is not covered by $[C, f_{\delta}(\widetilde{C}) \cap [n]]$.
\end{prop}

\begin{proof} 
Suppose that $[C, f_{\delta}(\widetilde{C}) \cap [n]]$ and $[D, f_{\eta}(\widetilde{D}) \cap [n]]$ have a nontrivial intersection, then $(C \cup D) \in [C, f_{\delta}(\widetilde{C}) \cap [n]] \cap [D, f_{\eta}(\widetilde{D}) \cap [n]]$.
If $D$ is not covered by $[C, f_{\delta}(\widetilde{C}) \cap [n]]$, then $C \nsubseteq D$. So there exists a $c \in C$ such that $c \in \mathscr{G}_{\eta}(\widetilde{D})$. 

Let $B_1, G_1, B_2, \cdots, B_p, G_p$ be the extended block structure of $\widetilde{D}$ on $[m'']$ and, without loss of generality, suppose $c \in G_1$. Note that the gaps$_{\eta}(\widetilde{D})$ on $[m']$ keeps unchanged in the extended block structure on $[m'']$ and

$$|\mathscr{G}_{\eta}(\widetilde{D})|=|f_{\eta} (\widetilde{D})|-|\widetilde{D}| \le \eta-1 \le \delta -1.$$

Denote $C'$ the block in blocks$_\delta (\widetilde{C})$ that contains $c$. Let $B_w$ be the block in 
the extended block structure of $\widetilde{D}$ on $[m'']$ that contains $\{n+1, \cdots, m''\}$. Repeating the same argument as that in the proof of Proposition $3.4$, we can show that $D \cap B_v \subseteq C$ and $B_v \subseteq C'$ for any $1<v<w$. For the block $B_w$, it is easy to see that $D \cap B_w \cap [n] \subseteq C$, $(\widetilde{D} \cap B_w) \cap [c, m''] \subseteq \widetilde{C}$ and $B_w \cap [c, m''] \subseteq C'$. Proceeding clockwise, suppose that $d_{w_1}$ is the first element in $D_w \cap D$ but not in $[c, m'']$. Then the inequality $[n-(d+q)] \delta \ge [n-(d+l)] \eta +(m''-m')$, which is equivalent to $(d+l+1) \eta \ge (d+q+1) \delta$, guarantees that $d_{w_1} \in C'$. Because $D \subseteq f_{\delta}(\widetilde{C}) \cap [n]=C \cup \mathscr{G}_{\delta}(\widetilde{C})$, $d_{w_1} \in C$. Continue this process, we can show that $D \cap B_w \subseteq C$ and $B_w \subseteq C'$.

Proceeding clockwise and using analogous arguments, we conclude that $D \cap \mathscr{B}_{\eta} (\widetilde{D}) \subseteq C$. However, $D \cap \mathscr{B}_{\eta} (\widetilde{D})=D$, so $D \cap \mathscr{B}_{\eta} (\widetilde{D}) \subseteq C$ and $|C| \le |D|$ imply that $C=D$, contradicting the fact that $C \nsubseteq D$. Therefore $[C, f_{\delta}(\widetilde{C}) \cap [n]]$ does not intersect $[D, f_{\eta}(\widetilde{D}) \cap [n]]$.
\end{proof}

Now we are ready to prove Theorem $1.2$. 

\section{Proof of Theorem 1.2}

Because Cimpoea\c{s} $[7]$ and Keller, Shen, Streib and Young $[13]$ have obtained an upper bound $\sdepth(I_{n,d}) \le \left \lfloor \binom{n}{d+1}/\binom{n}{d} \right\rfloor+d$, to prove Theorem $1.2 (1)$, it is sufficient to show that $\sdepth(I_{n,d}) \ge \left \lfloor \binom{n}{d+1}/\binom{n}{d} \right\rfloor+d$ if $1 \le d \le n \le (d+1) \left \lfloor \frac{1+\sqrt{5+4d}}{2}\right\rfloor+2d$.

For any $n \ge d$, $\sdepth(I_{n,d}) \ge d$ follows by taking trivial interval partitions of $P_{I_{n,d}}$. So for $d \le n \le 2d$, we have $\sdepth(I_{n,d}) \ge d=\left \lfloor \frac{n-d}{d+1}\right\rfloor+d= \left \lfloor \binom{n}{d+1}/\binom{n}{d} \right\rfloor+d$.

Throughout, we assume that $n \ge 2d+1$. Any such an $n$ can be written uniquely as $n=(d+1)k+d+r$ with $0 \le r \le d$ and $k \ge 1$.

When $k=1$, $l$ $(0 \le l<k)$ can only take a single value $0$. In this case we can take $s=1$ and $m=(n+1) \cdot 1+n$ to construct $\mathscr{I}_{n,d,2}$. An interval partition $\mathscr{P}$ of $P_{I_{n,d}}$ can be constructed as follows. First we include all the intervals in $\mathscr{I}_{n,d, 2}$ into $\mathscr{P}$. By Proposition $3.3$, the intervals in $\mathscr{I}_{n,d,2}$ are disjoint and their right end points have cardinality $d+1=d+k$. The remaining uncovered subsets of $P_{I_{n,d}}$ can be covered by trivial intervals because all of them have cardinality at least $d+1=d+k$.
This proves Theorem $1.2(1)$ for $2d+1 \le n \le 3d+1$.

When $k=2$, $l$ can take two values $0$ and $1$. Take $s=k-l$ and $m=(n+1)s+n$ in each case, we can construct $\mathscr{I}_{n,d,3}$ and $\mathscr{I}_{n,d+1,2}$ respectively. We then construct an interval partition $\mathscr{P}$ of $P_{I_{n,d}}$ as follows. First we include all the intervals in $\mathscr{I}_{n,d, 3}$ into $\mathscr{P}$. If the left endpoint of an interval in $\mathscr{I}_{n,d+1,2}$ is not covered by any element in $\mathscr{I}_{n,d, 3}$, then add it into $\mathscr{P}$; otherwise discard it. By Proposition $3.3$, the selected intervals in $\mathscr{P}$ are disjoint and their right end points have cardinality $d+2=d+k$. The remaining uncovered subsets of $P_{I_{n,d}}$ can be covered by trivial intervals because all of them have cardinality at least $d+2=d+k$. This proves Theorem $1.2(1)$ for $3d+2 \le n \le 4d+2$.

If $\left \lfloor \frac{1+\sqrt{5+4d}}{2}\right\rfloor \le 2$, Theorem $1.2(1)$ follows. So we may assume $\left \lfloor \frac{1+\sqrt{5+4d}}{2}\right\rfloor \ge 3$.

For any $n$ with $4d+3=(d+1)3+d \le n \le (d+1) \left \lfloor \frac{1+\sqrt{5+4d}}{2}\right\rfloor+2d$ , we can write $n$ as $n=(d+1)k+d+r$ with $3 \le k \le \left \lfloor \frac{1+\sqrt{5+4d}}{2}\right\rfloor$ and $0 \le r \le d$. It is easy to show that $\frac{n+1}{d+l+1} \ge k+1$ when $l=0$ and $\frac{n+1}{d+l+1} \ge k$ when $0 < l<k$. Taking $s=k$ when $l=0$ and $s=k-1$ when $0<l<k$ and letting $m=(n+1)s+n$, we can construct $\mathscr{I}_{n,d,k+1}$ and $\mathscr{I}_{n,d+l,k}$ for each $0< l<k$ respectively. Now we can construct an interval partition $\mathscr{P}$ of $P_{I_{n,d}}$ by selecting suitable intervals from $\mathscr{I}_{n,d,k+1}$ and $\mathscr{I}_{n,d+l,k}$.

We will build up $\mathscr{P}$ step by step. First we include all the intervals in $\mathscr{I}_{n,d,k+1}$ into $\mathscr{P}$ and denote the set of those intervals $\mathscr{P}_0$. Then we proceed to $\mathscr{I}_{n,d+1,k}$ to construct $\mathscr{P}_1$. If the left endpoint of an interval in $\mathscr{I}_{n,d+1,k}$ is not covered by any element in $\mathscr{P}_0$, then add it into $\mathscr{P}_0$; otherwise discard it. The resulting set will be denoted as $\mathscr{P}_1$. Continue this process. An interval in $\mathscr{I}_{n,d+l,k}$ will be added into $\mathscr{P}_{l-1}$ if its left endpoint is not covered by any element in $\mathscr{P}_{l-1}$. And the resulting set will be denoted as $\mathscr{P}_l$. After this selection process reaches its end at $l=k-1$, adding the remaining uncovered subsets of $P_{I_{n,d}}$ as trivial intervals into $\mathscr{P}_{k-1}$ gives us $\mathscr{P}$. We have the following theorem.

\begin{thm} For any $n$ and $d$ with $(d+1)3+d \le n \le (d+1) \left \lfloor \frac{1+\sqrt{5+4d}}{2}\right\rfloor+2d$, the Stanley depth $\sdepth(I_{n,d})=\left \lfloor \binom{n}{d+1}/\binom{n}{d} \right\rfloor+d$.
\end{thm}
\begin{proof}
It is obvious that $P_{I_{n,d}}$ is covered by $\mathscr{P}$. The hard part is to show that the intervals in $\mathscr{P}$ are disjoint.

By Proposition $3.3$, the intervals in $\mathscr{P}_0$ are disjoint. Again by Proposition $3.3$, the intervals in $\mathscr{P}_1$ are disjoint. Suppose the intervals in $\mathscr{P}_{l-1}$ are disjoint. If $l=k$, we are done. Otherwise, it is sufficient to prove that the intervals in $\mathscr{P}_{l} \setminus \mathscr{P}_{l-1}$ are pairwise disjoint, and moreover, any interval in $\mathscr{P}_{l} \setminus \mathscr{P}_{l-1}$ is disjoint with any one in $\mathscr{P}_{l-1}$. Because all the intervals in $\mathscr{P}_{l} \setminus \mathscr{P}_{l-1}$ are from $\mathscr{I}_{n,d+l,k}$, they are pairwise disjoint by Proposition $3.3$. So we only need to show that any interval in $\mathscr{P}_{l} \setminus \mathscr{P}_{l-1}$ cannot have a nontrivial intersection with any one in $\mathscr{P}_{l-1}$. 

Suppose that the left endpoint $D$ of an interval $[D,f_{k}(\widetilde{D}) \cap [n]]$ in $\mathscr{I}_{n,d+l,k}$ is not covered by any element in $\mathscr{P}_{l-1}$. By Proposition $3.3$, it does not intersect any interval in $\mathscr{I}_{n,d,k+1}$. For any remaining interval $[C,f_{k}(\widetilde{C}) \cap [n]]$ in $\mathscr{P}_{l-1}$, Proposition $3.5$ guarantees that $[C,f_{k}(\widetilde{C}) \cap [n]]$ and $[D,f_{k}(\widetilde{D}) \cap [n]]$ do not intersect. Therefore any interval in $\mathscr{P}_{l} \setminus \mathscr{P}_{l-1}$ cannot have a nontrivial intersection with any one in $\mathscr{P}_{l-1}$. This proves that the intervals in $\mathscr{P}$ are disjoint.  

The right endpoint of any interval in $\mathscr{P}$ has cardinality at least $d+k$, so $\sdepth (I_{n,d}) \ge d+k$. Combining this inequality with Lemma $2.2$ in $[13]$ or Theorem $1.1 (b)$ in $[7]$, we have $\sdepth(I_{n,d})=d+k$.
\end{proof}

Combining the results for $k=0, 1, 2$ and Theorem $4.1$, we complete the proof of Theorem $1.2 (1)$.

For any $n > (d+1) \left \lfloor \frac{1+\sqrt{5+4d}}{2}\right\rfloor+2d$, $n$ can be written uniquely as $n=(d+1)k+d+r$ with $0 \le r \le d$ and $k \ge \left \lfloor \frac{1+\sqrt{5+4d}}{2}\right\rfloor+1 \ge 3$. 
Let $s=\left \lfloor \frac{-(d+2)+\sqrt{d^2+4(n+1)}}{2} \right\rfloor$; it satisfies the inequality $s+1 \le \frac{n+1}{d+s+1}$. Thus for any integer $q$ with $1 \le q \le s$, we have $s+1 \le \left \lfloor \frac{n+1}{d+q+1}\right\rfloor$, and $(s+1)$ can be used as a common density for constructing $\mathscr{I}_{n,d+q,s+1}$.
An interval partition $\mathscr{P}$ of $P_{I_{n,d}}$ can be construct as follows.

First we include all the intervals in $\mathscr{I}_{n,d,k+1}$ into $\mathscr{P}$ and denote the set of these intervals $\mathscr{P}_0$. Then we proceed to $\mathscr{I}_{n,d+1,s+1}$ to construct $\mathscr{P}_1$. If the left endpoint of an interval in $\mathscr{I}_{n,d+1,k}$ is not covered by any element in $\mathscr{P}_0$, then add it into $\mathscr{P}_0$; otherwise discard it. The resulting set will be denoted as $\mathscr{P}_1$. Continue this process. An interval in $\mathscr{I}_{n,d+q,s+1}$ will be added into $\mathscr{P}_{q-1}$ if its left endpoint is not covered by any element in $\mathscr{P}_{q-1}$. After this selection process reaches $s=\left \lfloor \frac{-(d+2)+\sqrt{d^2+4(n+1)}}{2} \right\rfloor$, adding the remaining uncovered subsets of $P_{I_{n,d}}$ as trivial intervals into $\mathscr{P}_{s}$ gives $\mathscr{P}$. The disjointness of the intervals in $\mathscr{P}$ guarantees by Proposition $3.3$ and Proposition $3.5$. Repeating the same argument as that in the proof of Theorem $3.6$, we can show that $\mathscr{P}$ is an interval partition of $P_{I_{n,d}}$. Now $d+1+s=d+1+\left \lfloor \frac{-(d+2)+\sqrt{d^2+4(n+1)}}{2} \right\rfloor=\left \lfloor \frac{d+\sqrt{d^2+4(n+1)}}{2} \right\rfloor \le \sdepth(I_{n,d})$ follows. The upper bound has been obtained in $[7]$ and $[13]$. This completes the proof of Theorem $1.2(2)$.

\section{An alternative proof of Theorem $1.1$ in $[13]$}

Our construction leads to a direct proof of Theorem $1.1$ in $[13]$, without using graph theory. In order to prove the first part of this theorem, we need the following lemma.

\begin{lem} Fixed a nonnegative integer $k$, and suppose that $\sdepth(I_{(d+1)k+d,d})=d+k$ for any positive integer $d$. Then for any positive integers $n$ and $d$, such that $n \ge (d+1)k+d$, we have $\sdepth(I_{n,d}) \ge d+k$, in particular, $\sdepth(I_{n,d})=d+k$ for $(d+1)k+d \le n \le (d+1)k+2d$.
\end{lem}
\begin{proof}
We use double induction, first on $d$ and then on $n$.

When $d=1$, Biro et al. $[4]$ proved that $\sdepth(I_{n,d}) \ge \left \lceil \frac{n}{2} \right\rceil$. So if $n \ge 2k+1$, we have 

$$\sdepth(I_{n,d}) \ge \left \lceil \frac{n}{2} \right\rceil \ge \left \lceil \frac{2k+1}{2} \right\rceil=1+k=d+k.$$

For the induction step, suppose for all $1 \le d' <d$, we have $\sdepth(I_{n,d'}) \ge d'+k$ for any $n \ge (d'+1)k+d'$. Now consider $\sdepth(I_{n,d})$. For this, we use induction on $n$.

The smallest value for $n$ is $(d+1)k+d$. In this case $\sdepth(I_{n,d}) = d+k$ by the assumption in Lemma $5.1$. For any $n > (d+1)k+d$, suppose $\sdepth(I_{n-1,d}) \ge d+k$ for $n-1 \ge (d+1)k+d$. By induction assumption on $d$, for $n-1 \ge (d+1)k+d> (d-1+1)k+d-1$, we have that $\sdepth(I_{n-1,d-1}) \ge d-1+k$. Therefore by Lemma $2.5$ in $[13]$, we have $\sdepth(I_{n,d}) \ge d+k$ for any $n > (d+1)k+d$.

By Lemma $2.2$ in $[13]$ or Theorem $1.1 (b)$ in $[7]$, $\sdepth(I_{n,d}) \le \left \lfloor \binom{n}{d+1}/\binom{n}{d} \right\rfloor+d=\left \lfloor \frac{n-d}{d+1} \right\rfloor+d=k+d$ when $(d+1)k+d \le n \le (d+1)k+2d$. Combining this with $\sdepth(I_{n,d}) \ge d+k$ for any $n \ge (d+1)k+d$, we immediately have $\sdepth(I_{n,d}) = d+k$ for $(d+1)k+d \le n \le (d+1)k+2d$.
\end{proof}

The notation $c$ used in $[13]$ is equal to $k+1$ here. So $c=1,2,3,4$ correspond to $k=0, 1, 2, 3$ respectively.

For $k=0, 1, 2$, we have proved in Section $4$ that $\sdepth(I_{(d+1)k+d,d})=d+k$ for any positive integer $d$. For $k=3$, Theorem $4.1$ is not sufficient; it misses the four cases when $d=1,2,3,4$. Instead, we prove $\sdepth(I_{(d+1)3+d,d})=d+3$ for any $d$ more directly.

By Lemma $3.3$ in $[13]$, each $(d+1)$-subset in $P_{I_{n,d}}$ is covered by $\mathscr{I}_{n,d,k+1}=\mathscr{I}_{n,d,4}$. For the $(d+2)$-subsets, we can construct $\mathscr{I}_{n,d+l,k-l+1}=\mathscr{I}_{n,d+2,2}$. An interval partition $\mathscr{P}$ of $P_{I_{(d+1)3+d, d}}$ can be constructed as follows. First we include all the intervals in $\mathscr{I}_{n,d,4}$ into $\mathscr{P}$. If the left endpoint of an interval in $\mathscr{I}_{n,d+2,2}$ is not covered by any element in $\mathscr{I}_{n,d,4}$, add this interval into $\mathscr{P}$; otherwise discard it. Adding the remaining uncovered subsets of $P_{I_{(d+1)3+d, d}}$ as trivial intervals into $\mathscr{P}$. It is easy to show that $\mathscr{P}$ gives an interval partition of $P_{I_{n,d}}$ by Proposition $3.3$. The right endpoint of any interval in $\mathscr{P}$ has cardinality at least $d+3$, so $\sdepth(I_{(d+1)3+d,d}) \ge d+3$. Combining this with the upper bound $\sdepth(I_{(d+1)3+d,d}) \le \left \lfloor \binom{(d+1)3+d}{d+1}/\binom{(d+1)3+d}{d} \right\rfloor+d=\left \lfloor \frac{(d+1)3+d-d}{d+1} \right\rfloor+d=3+d$ gives $\sdepth(I_{(d+1)3+d,d})=d+3$. This finishes the proof of Theorem $1.1(1)$ by Lemma $5.1$.

The left inequality in Theorem $1.1 (2)$ follows from Lemma $5.1$ because the Stanley depth $\sdepth(I_{(d+1)3+d,d})=d+3$ for any positive integer $d$. The right inequality in Theorem $1.1(2)$ has been proved in $[13]$ and $[7]$. Combining these two inequalities gives Theorem $1.1(2)$.

\bigskip
\begin{center}  
{\bf References}
\end{center}
\medskip

\begin{enumerate}[itemsep=0pt, parsep=0pt]

\item I. Anwar and D. Popescu, Stanley Conjecture in small embedding dimension, J. Algebra ${\bf 318}$, $1027$-$1031$, $2007$.

\item J. Apel, On a conjecture of R.P. Stanley. I. Monomial ideals, J. Algebraic Combin. ${\bf 17}$ $(1)$, $39$-$56$, $2003$.

\item J. Apel, On a conjecture of R.P. Stanley. II. Quotients modulo monomial ideals, J. Algebraic Combin. ${\bf 17}$ $(1)$, $57$-$74$, $2003$.

\item C. Bir\'o, D. Howard, M. Keller, W. Trotter and S. Young, Interval partition and Stanley depth. To appear in J. Combin. Theory Ser. A. $doi:10.1016/j.jcta.2009.07.008,$

\noindent $2009$.

\item M. Cimpoea\c{s}, Stanley depth of complete intersection monomial ideals. Bull. Math. Soc. Sci. Math. Roumanie (N.S.) ${\bf 51}$ $(99)$, no. $3$, $205$-$211$, $2008$.

\item M. Cimpoea\c{s}, A note on Stanley's conjecture for monomial ideals. $arXiv:0906.1303$

\noindent $[math.AC], 2009$.

\item M. Cimpoea\c{s}, Stanley depth of square free Veronese ideals. $arXiv:0907.1232$

\noindent $[math.AC], 2009$.

\item A. Dress, A new algebraic criterion for shellability. Beitrage Algebra Geom. ${\bf 34}$ $(1)$, 

\noindent $45$-$55$, $1993$.

\item J. Herzog and T. Hibi, Cohen-Macaulay polymatroidal ideals. European J. Combin. ${\bf 27}$, $513$-$517$, $2006$.

\item J. Herzog, A. Jahan and S. Yassemi, Stanley decompositions and partitionable simplicial complexes. J.
Algebraic Combin. ${\bf 27}$, $113$-$125$, $2008$.

\item J. Herzog, M. Vladoiu and X. Zheng, How to compute the Stanley depth of a monomial ideal.  J. Algebra ${\bf 322}$ $(9)$, $3151$-$3169$, $2009$.

\item A. Jahan, Prime filtrations of monomial ideals and polarizations, J. Algebra ${\bf 312}$ $(2)$,$1011$-$1032$, $2007$.

\item M. Keller, Y. Shen, N. Streib and S. Young, On the Stanley depth of squarefree Veronese ideals, $arXiv:0910.4645v1 [math.AC], 2009$.

\item M. Keller and S. Young, Stanley depth of squarefree monomial ideals. J. Algebra ${\bf 322}$ $(10)$, $3789$-$3792$, $2009$.

\item S. Nasir, Stanley decompositions and localization. Bull. Math. Soc. Sci. Math. Roumanie (N.S.) ${\bf 51}$ $(99)$ no. $2$, $151$-$158$, $2008$.

\item R. Okazaki, A lower bound of Stanley depth of monomial ideals. To appear in J. Commut. Algebra, $2009$.

\item D. Popescu, Stanley depth of multigraded modules. J. Algebra ${\bf 321}$ $(10)$, $2782$-$2797$,

\noindent $2009$.

\item A. Rauf, Stanley decompositions, pretty clean filtrations and reductions modulo regular elements, Bull. Soc. Sc. Math. Roumanie ${\bf 50}$ $(98)$, no. $4$, $347$-$354$, $2007$.

\item Y. Shen, Stanley depth of complete intersection monomial ideals and upper-discrete partitions. J. Algebra
${\bf 321}$ $(4)$, $1285$-$1292$, $2009$.

\item R. Stanley, Linear Diophantine equations and local cohomology. Invent. Math. ${\bf 68}$, $175$-$193$, $1982$.

\end{enumerate}

{\footnotesize DEPARTMENT OF MATHEMATICS, ANHUI UNIVERSITY, HEFEI, ANHUI, 230039, CHINA}

$\textit{E-mail address:}$ ge1968@126.com

\medskip

{\footnotesize DEPARTMENT OF MATHEMATICS, SUNY CANTON, 34 CORNELL DRIVE, CANTON, 

\hskip .05 cm NY 13617, USA}

$\textit{E-mail address:}$ linj@canton.edu

\medskip

{\footnotesize DEPARTMENT OF MATHEMATICS, UNIVERSITY OF SCIENCE AND TECHNOLOGY OF   

\hskip .05 cm CHINA,HEFEI, ANHUI, 230026, CHINA}

$\textit{E-mail address:}$ yhshen@ustc.edu.cn

\end{document}